\documentclass[preprint,12pt]{elsarticle}

\usepackage{a4wide,amsfonts, amsmath, amscd}
\usepackage[psamsfonts]{amssymb}

\usepackage{graphicx}
\usepackage[cp1251]{inputenc}
\usepackage{amssymb}
\usepackage{amsthm}
\usepackage{cmap}

\usepackage{amsmath}

\biboptions{sort&compress}

\newcommand{\sysn}{\left\{\begin{array}{rcl}}
\newcommand{\sysk}{\end{array}\right.}

\renewcommand{\le}{\leqslant}
\renewcommand{\ge}{\geqslant}

\newtheorem{theorem}{Theorem}[section]
\newtheorem{lemma}[theorem]{Lemma}

\theoremstyle{example}

\newtheorem{proposition}[theorem]{Proposition}
\theoremstyle{definition}
\newtheorem{definition}[theorem]{Definition}
\newtheorem{remark}[theorem]{Remark}
\newtheorem{corollary}[theorem]{Corollary}

\journal{...}

\begin{document}

\title{Generalization of the Grothendieck's theorem}
\author[affil1]{Mikhail Al'perin}

\address[affil1]{Krasovskii Institute of Mathematics and Mechanics, Yekaterinburg, Russia}

\ead{alper@mail.ru}

\address[affil2]{Ural Federal University, Yekaterinburg, Russia}

\author[affil1,affil2]{Alexander V. Osipov}


\ead{OAB@list.ru}

\begin{abstract} In this paper, we obtain a generalization of Grothendieck's theorem for the space of continuous mappings $C_{\lambda,\mu}(X,Y)$  where $Y$ is a complete uniform space with the uniformity $\mu$ endowed with the topology of uniform convergence on the family $\lambda$ of subsets  of $X$.
A new topological game is defined - the Asanov-Velichko game, which makes it possible to single out a class of topological spaces of the Grothendieck type.

The developed technique is used to generalize the Grothendieck's theorem for the space of continuous mappings endowed with the set-open topology.

\end{abstract}

\tnotetext[label1]{The research of the second author was supported by the Russian Science Foundation (RSF Grant No. 23-21-00195).}


\maketitle 


\section{Introduction}

In 1952 Grothendieck \cite{Grot} proved the following result.

\begin{theorem}(Grothendieck) Let $X$ be a compact space and let $Y$ be a metrizable
space. Then each relatively countably compact subspace
 of $C_p(X,Y)$ is relatively compact.
\end{theorem}

This theorem has played an important role in topology and
functional analysis.

Grothendieck's theorem has been generalized many times \cite{Arch1,Chr,PS,Pt}. The following successful generalization of this theorem is due to M.O. Asanov and N.V. Velichko \cite{AsVel}.

\begin{theorem} If a space $X$ is countably compact, then each bounded subset of $C_p(X)$ is relatively compact.
\end{theorem}

Let us recall two theorems proved by A.V. Arhangel'skii \cite{arch}.

\begin{theorem}\label{1.8} If a Tychonoff space $X$ contains a
dense $\sigma$-countably pracompact subspace $Y$, then
every pseudocompact subspace $P$ of $C_p(X)$ is an Eberlein compactum.
\end{theorem}

\begin{theorem}\label{1.9} If a Tychonoff space $X$ contains a
dense $\sigma$-pseudocompact subspace $Y$,
then every countably compact subspace $P$ of $C_p(X)$ is an Eberlein compactum.
\end{theorem}

In this paper, we investigate a generalization of the Grothendieck's theorem for the space of continuous mappings $C_{\lambda,\mu}(X,Y)$  where $Y$ is a complete
uniform space with uniformity $\mu$ endowed with the topology of uniform convergence on the family $\lambda$ of subsets of $X$. Also, the developed technique is used to generalize the Grothendieck's theorem for the space of continuous mappings $C_{\lambda}(X, \mathbb{R})$ endowed with the set-open ($\lambda$-open) topology.

\section{Notation and terminology}
 The set of positive integers is denoted by $\mathbb{N}$ and
$\omega=\mathbb{N}\cup \{0\}$. Let $\mathbb{R}$ be the real line,
we put $\mathbb{I}=[0,1]\subset \mathbb{R}$, and let $\mathbb{Q}$
be the rational numbers. We denote by $\overline{A}$ (or $Cl_X A$)
the closure of $A$ (in $X$).

Let $X$ be a topological space. Let $(Y,\mu)$ be a uniform space and $\lambda\subseteq 2^X$.
A topology on $C(X,Y)$ generated by the
uniformity
$$\nu=\{\langle A,M\rangle \subseteq C(X,Y)\times
C(X,Y):A\in\lambda,M\in\mu\}$$ where
$$\langle A,M\rangle = \{\langle f,g\rangle \in C(X,Y)\times
C(X,Y):\forall x\in A~\langle f(x),g(x)\rangle\in M\}$$ is called the
{\it topology of uniform convergence on elements of $\lambda$} and
denote by $C_{\lambda,\mu}(X,Y)$.

The well-known fact is that if $\lambda$ is the family of all compact
subsets of $X$ or all finite subsets of $X$ then the topology on
$C(X,Y)$ induced by the uniformity $\nu$ of uniform convergence on
elements of $\lambda$ depends only on the topology induced on $Y$
by the uniformity $\mu$ (see \cite{Eng}). In these cases, we
use the notation $C_c(X,Y)$ and $C_p(X,Y)$, respectively. If
$Y=\mathbb{R}$ then $C_c(X)$ and $C_p(X)$, respectively.

In case $(Y,\rho)$ is a metric space and the uniformity $\mu$
is induced by the metric $\rho$, then for $C_{\lambda,\mu}(X,Y)$,
we use the notation $C_{\lambda,\rho}(X,Y)$ and
$C_{\lambda,\rho}(X)$ for the case $Y=\mathbb{R}$.

If $X\in\lambda$, we write $C_\mu(X,Y)$ in place of
$C_{\lambda,\mu}(X,Y)$ and $C_{\mu}(X)$ in place of
$C_{\mu}(X,\mathbb{R})$.

\begin{remark}\label{2.1}
For the topology of uniform convergence on elements of $\lambda$,
we assume that the following natural conditions hold:

(1) if $A\in\lambda$ and $A'\subseteq A$ then $A'\in\lambda$.

(2) if $A_1, A_2\in\lambda$ then $A_1\bigcup A_2\in\lambda$.

(3) if $A\in\lambda$ then $\overline{A}\in\lambda$.

 Note that $C_{\lambda,\mu}(X,Y)$ is Hausdorff if and only if the set $\bigcup \lambda$ is dense in $X$,
 i.e.,
\hbox{$\overline{\bigcup\{A:A\in\lambda\}}=X$} (see
\cite{Bur75b}).

 Then we have an additional condition on $\lambda$.

(4) $\lambda$ is a cover of $X$.

\end{remark}

The {\it set-open topology} on a family $\lambda$ of non-empty
subsets of the set $X$ is a generalization of the compact-open
topology and of the topology of pointwise convergence. This
topology, first introduced by Arens and Dugunji in \cite{AreDug}, is
one of the important topologies on $C(X,Y)$.

If $A\subseteq X$ and $V\subseteq Y$, then $[A,V]$ is defined
as $[A,V]=\{f\in C(X,Y): f(A)\subseteq V\}$.

Let $X$ and $Y$ be topological spaces, $\lambda\subseteq 2^X$ . A
topology on $C(X,Y)$ is called a {\it $\lambda$-open topology}
(set-open topology) if the family $\{[A,V]: A\in \lambda$
and $V$ is open in $Y \}$ form a subbase for the topology. The
function space $C(X,Y)$ with this topology is denoted
by $C_{\lambda}(X,Y)$.

Recall that a subset $A$ of a topological space $X$ is called

$\bullet$ {\it relatively compact} if $A$ has a compact closure in $X$.

$\bullet$ {\it bounded} if every continuous function on $X$ is bounded on $A$.

$\bullet$  {\it relatively countably compact} if each sequence of $A$ has a
limit point in $X$.

$\bullet$  {\it
pseudocompact} if any continuous real-valued function on $A$ is bounded.

\medskip

A topological space $X$ is called  {\it countably
pracompact} if there is a subspace $Y\subseteq X$ which is dense in $X$ and countably compact in $X$ in the following sense: every infinite set $A\subseteq Y$ has a limit point in $X$.

\medskip

A space is called $\sigma$-compact ($\sigma$-countably compact, etc.) if it is the union of a countable set of compact (respectively, countably compact, etc.) subspaces.

\medskip

{\it A $G_\delta$-neighborhood of a point $x$} in a space
$X$ is called a set $V$ containing $x$ which is the
intersection of a countable number of open sets.

For other notation and terminology almost without exceptions we follow the Engelking's book \cite{Eng}.

\section{Asanov-Velichko game and compactness in
$C_p(X,Y)$}

Let $X$ be a topological space, let $\lambda$ be a family of subsets
of $X$ and let $A$ be a non-closed subset of $X$ and $x\in
\overline{A}\setminus A$.

The following game on the space $X$ will be called the {\it Asanov--Velichko game generated by the family $\lambda$}.

Player ONE chooses at the $n$-th step of the game a neighborhood $V_n$ of the point
$x$. Player TWO chooses at the $n$-th step of the game a subset $S_n$ of
the set $A$ which lies in the closure of some $T_n\in
\lambda$. The game is played with a countable number of steps. We say that
the game is won by player TWO if the set $\overline{\bigcup\{S_n:n\in
\mathbb{N}\}}\cap (\bigcap\{V_n: n\in \mathbb{N}\})\ne \emptyset$
and --- by player ONE otherwise.  This game will be denoted by $AV_{\lambda}(X)$.

\medskip

The following definition generalizes the notions of weakly $p$- and
weakly $q$-spaces introduced in \cite{AsVel}.

\begin{definition} A space $X$ is called {\it Asanov-Velichko generated by the family $\lambda$ ($AV_{\lambda}$-space)} if for any non-closed subset $A$ of $X$
there is a point $x\in \overline{A}\setminus A$ for which player TWO
 has a winning strategy in the game $AV_{\lambda}(X)$.
\end{definition}

\medskip

In what follows, we will need two definitions.

\begin{definition} Let $X$ be a topological space.
A family $\lambda$ of subsets of $X$ will be called {\it countably
invariant} if $\lambda$ contains all singletons of $X$ and the fact that $A_n\in \lambda$ for every $n\in
\mathbb{N}$ implies $\bigcup \{A_n:n\in \mathbb{N}\}\in
\lambda$.
\end{definition}

\begin{definition}(\cite{arch}) Property ${\cal P}$ of subsets of a
topological space $X$ is called a {\it continuously invariant property}
if the facts that $A$ is a subset of $X$ with the property ${\cal P}$ and $f:X\rightarrow Y$ is a continuous map imply that $f(A)$ has the property
${\cal P}$ in $f(X)$.
\end{definition}

The following proposition makes it easy to prove generalizations of the Grothendieck's theorem.

\begin{proposition}\label{1.5} Let $X$ be an $AV_{\lambda}$-space where $\lambda$ is
a countably invariant family of subsets of $X$.
Let ${\cal P}$ be a continuously invariant property of
subsets of $X$, and the following condition is satisfied:

$(\gamma )$  for every $Z\in \lambda$ and every $B\subseteq \pi
_Z(C_p(X))$ with the property ${\cal P}$, $B$ is compact.

Then for every $F\subseteq C_p(X)$, having the property ${\cal
P}$, $F$ is relatively compact.
\end{proposition}

\begin{proof} Note that to prove the proposition
it is enough to prove that
$\overline{F}^{C_p(X)}=\overline{F}^{R^{X}}$. Indeed, for
any $x\in X$ the set $\pi_x(F)$ is compact and, therefore,
$\overline{F}^{R^{X}}\subseteq \prod\{\pi_x(F):x\in X\}$ is
compact.

Suppose $\overline{F}^{C_p(X)}\neq
\overline{F}^{\mathbb{R}^{X}}$, that is, there is a discontinuous function
$f\in \overline{F}^{\mathbb{R}^{X}}$. Then there is
a closed subset $M$ of $\mathbb{R}$ such that $C=f^{-1}(M)$
 is not closed in $X$. Since $X$ is an $AV_{\lambda}$-space, there is $x_0\in \overline{C}\setminus C$,
for which player TWO has a winning strategy in the game $AV_{\lambda}(X)$. Let $\epsilon =\rho
(f(x_0),M)$, then $\epsilon >0.$

The facts that $\lambda$ is countably invariant and the condition $(\gamma)$
imply that $\pi_{x_0}(F)$ is compact. Therefore, there is
$f_1\in F$ such that $\pi_{x_0}(f)=\pi_{x_0}(f_1)$, i.e.,
$f_1(x_0)=f(x_0)$. Let $V_1=f^{-1}_1(\{r\in \mathbb{R}:\rho
(f(x_0),r)<1\})$. Let $S_1\subseteq C$ and $T_1\in \lambda$ be the
sets selected by the player TWO according to the winning
strategy in the game $AV_{\lambda}(X)$.

Suppose that the functions $f_1,\ldots ,f_n$,
the neighborhoods $V_1,\ldots ,V_n$ of $x_0$, the subsets $S_1,\ldots
,S_n$ of $C$,
  and the elements $T_1,\ldots ,T_n$ of $\lambda$,
such that $S_i\subseteq \overline{T_i}$ for every $i=1,\ldots,n$ are constructed.
By condition $(\gamma)$, the projection of the set $F$ onto
$L=\bigcup\{T_i:i=1,\ldots ,n\}\cup \{x_0\}$ is compact,
hence, there is $f_{n+1}\in F$ such that
$\pi_L(f)=\pi_L(f_{n+1})$. Let $V_{n+1}=f^{-1}_{n+1}(\{r\in
\mathbb{R}:\rho(f(x_0),r)<\frac{1}{n+1}\})\cap V_n$. Let
$S_{n+1}\subseteq C$ and $T_{n+1}\in \lambda$ be the sets chosen by
player TWO.

Regarding the constructed families $\{f_i:i\in \mathbb{N}\}$, $\{V_i:i\in
\mathbb{N}\}$, $\{S_i:i\in \mathbb{N}\}$ and $\{T_i:i\in
\mathbb{N}\}$, since $X$ is an $AV_{\lambda}$-space, there is $x_{\infty}\in \overline{\bigcup
\{S_n:n\in \mathbb{N}\}}\cap (\bigcap \{V_n:n\in \mathbb{N}\})$.

Consider $Y=\bigcup \{T_i:i\in \mathbb{N}\}\cup
\{x_{\infty}\}\cup\{x_0\}$. By condition $(\gamma)$, $\pi_Y(F)$
is compact, that is, there exists $f_{\infty}\in F$ such that
$\pi_Y(f)=\pi _Y(f_{\infty})$. Then the following conditions are satisfied:

(1) $f_{\infty}(x_{\infty})=f(x_0)$;

(2) $f_{\infty}(x)=f(x)$ for every $x\in \bigcup \{T_i:i\in
\mathbb{N}\}$ because $\pi_Y(f)=\pi_Y(f_{\infty})$ and
$\pi_Y(f_{\infty})$ is a limit function for the set
$\pi_Y(\{f_i:i\in \mathbb{N}\})$.

Let $T=(\bigcup\{T_i:i\in \mathbb{N}\})\setminus \{x\in X:
\rho(f(x),f(x_{\infty}))<\frac{3}{4}\epsilon\}$.

 By conditions (1) and
(2), $\rho (f_{\infty}(x),f_{\infty}(x))\geq
\frac{3}{4}\epsilon$ for every $x\in T$. If we prove that
$x_{\infty}\in \overline{T}$, then we obtain a contradiction with the fact that
$f_{\infty}$ is continuous.

Indeed, let $O(x_{\infty})$ be an arbitrary neighborhood of the
point $x_{\infty}$, then there is $x'$ and $n\in \mathbb{N}$
such that $x'\in S_n\cap O(x_{\infty})$. Let $H=T_n\cup
\{x'\}$. By condition $(\gamma )$, $\pi_H(F)$ is compact, hence, there exists
$f'\in F$ such that $\pi_H(f)=\pi_H(f')$. Since $f'$ is continuous, there is a point $x^{''}\in T_n$ for which
$\rho(f'(x^{''}),f'(x'))<\frac{1}{4}\epsilon$, but then

$\rho (f(x^{''}),f(x_0))\geq\rho(f(x'),f(x_0))-\rho
(f'(x^{''}),f'(x'))\geq\epsilon -\frac{1}{4}\epsilon
=\frac{3}{4}\epsilon$, i.e., $x^{''}\in T\cap O(x_{\infty})$.

\end{proof}

Recall that a compactum $F$ is called an {\it Eberlein compactum} if there is a compactum $X$ such that $F$ is homeomorphic to a subspace of $C_p(X)$.

The following theorem is an obvious consequence of Theorem \ref{1.8} and
Proposition \ref{1.5}.

\begin{theorem}\label{1.10} Let $X$ be a Tychonoff  $AV_{\lambda_c}$-space where $\lambda_c$ is a family of all of its subspaces that have a dense $\sigma$-countably pracompact subspace. Then every pseudocompact subspace in $C_p(X)$
is relatively compact.
\end{theorem}

The following result also follows from Theorem \ref{1.9} and Proposition \ref{1.5} .

\begin{theorem}\label{1.11} Let $X$ be a Tychonoff $AV_{\lambda_p}$-space where $\lambda_p$ is a family of all of its subspaces with a dense $\sigma$-pseudocompact subspace. Then each countably pracompact subspace in $C_p(X)$ is relatively compact.
\end{theorem}

There is an interesting similarity between Theorems \ref{1.10}  and \ref{1.11} and the results (Theorems \ref{3.9} and \ref{3.10})
Archangel'skii in (\cite{arch} and \cite{Arch2}).

\begin{definition} Let $\lambda$ be a family
of subsets of a topological space $X$. The space $X$
is {\it functionally generated by the family $\lambda$} if the following condition holds: for every
discontinuous function $f:X\rightarrow \mathbb{R}$ there is
$A\in \lambda$ such that the function $\pi_A(f)$ cannot be extended to a real-valued continuous function on all of $X$.
\end{definition}

\begin{proposition} (Proposition 4.10 in \cite{arch}) Let a Tychonoff space $X$ be functionally generated by a family $\lambda$ of its subsets, let ${\cal P}$ be a continuously invariant property, and let the following condition hold:

$(\alpha)$ if $Y\in \lambda$, then for every $B\subseteq
\pi_Y(C_p(X))$ with the property ${\cal P}$, the closure $B$ in
$\pi_Y(C_p(X))$ is compact.

Then any subset of $C_p(X)$ with the property ${\cal
P}$, is relatively compact.
\end{proposition}

Although the condition $(\alpha)$ is somewhat different from the condition $(\gamma)$, similarity between Theorems \ref{1.10} and \ref{1.11} and two of the following theorems is striking.

\begin{theorem}(Theorem 8.1 in \cite{Arch2})\label{3.9}. If a Tychonoff space $X$ is functionally generated by a family $\lambda_c$
 of all of its closed subspaces that have a dense $\sigma$-countably pracompact
subspace, then the closure in $C_p(X)$ of every pseudocompact subspace is
compact.
\end{theorem}

\begin{theorem}(Theorem 8.3 in \cite{Arch2})\label{3.10}. If a Tychonoff space $X$ is functionally generated by a family $\lambda_p$ of all of its closed subspaces that contain a dense $\sigma$-pseudocompact subspace, then the closure in $C_p(X)$ of every countably pracompact subspace is compact.
 \end{theorem}

As we will see later, a certain similarity between the results
using functional generation and generation by
Asanov-Velichko can be observed.

\section{The Grothendieck's theorem for spaces
$C_{\lambda,\mu}(X,Y)$}

  In this section, we prove two
theorems that can be considered basic in this paper.

\medskip
Recall that a topological space $X$
is called a {\it $\mu$-space} if every bounded
subset of $X$ is relatively compact \cite{Buch}. The class of $\mu$--spaces arose in connection with the study of the question of
barrels of spaces of continuous
real-valued functions $C_c(X)$ with compact-open topology. In 1973 H. Buchwolter proved that $C_c(X)$ is barreled if and only if
$X$ is a $\mu$-space. Later the notion of $\mu$-space
acquired a value of its own. The class of $\mu$-spaces is very
wide, it includes all Dieudonn\'{e} complete spaces,
in particular, all metrizable spaces.

\begin{lemma}\label{2.2}
Let $X$ be a Tychonoff space and let $Y$ be a
Tychonoff $\mu$--space with a countable pseudocharacter. Let $f:X \rightarrow Y$ be a condensation (= a continuous bijection). Then $X$ is a $\mu$-space.
\end{lemma}

\begin{proof} Let $f_1:\nu X\rightarrow \nu Y$ be a
continuous extension of the function $f$ from the Hewitt real-compactification $\nu X$ of $X$  onto the Hewitt real-compactification $\nu Y$ of $Y$ (Theorem 3.11.16 in \cite{Eng})). Now let $A$ be a bounded subset of $X$,
then $\overline{A}^{\nu X}$ is compact (Proposition 6.9.7 in \cite{ArTk}). Let us prove that $\overline{A}^{X}=\overline{A}^{\nu X}$, i.e. $\overline{A}^{\nu
X}\setminus X=\emptyset$. Let this not be the case and there is a point $x\in
\overline{A}^{\nu X}\setminus X$. The set $f(A)$ is bounded in
$Y$ and $\overline{f(A)}^Y$ is compact, hence,
$f_1(\overline{A}^{\nu X})=\overline{f(A)}^Y\subset Y$. Then
$V=f^{-1}_1(f_1(x))$ is a $G_{\delta}$-neighbourhood of $x$ because the pseudocharacter $\psi(f_1(x),Y)$ of the point $f_1(x)$ in $Y$ is countable,  and, hence, $\psi(f_1(x),\nu
Y)$ is countable, too. Since the mapping $f$ is one-to-one, $V\cap X$ contains only one point $x_1$. Let $O(x)$ be a neighborhood of the point $x$ in $\nu X$ which does not contain $x_1$. But then $V_1=V\cap O(x)$ is a $G_\delta$-neighbourhood of the point $x$ and is disjoint with $X$, which contradicts the following assertion
(Theorem 3.11.11 in \cite{Eng}): for any point $x\in \nu X$ and for any of its neighborhoods
$V\subset \nu X$ the intersection of $V\cap X$ is not empty.
\end{proof}

\begin{definition} Let $X$ be a topological space.
Let $\lambda$ be a family of subsets of $X$. A subset $A$ of $X$ is called {\it $\lambda$-separable}
if there is a countable subfamily $\lambda_1$ of the family $\lambda$ such that $A\subseteq \overline{\bigcup\{B:B\in
\lambda_1\}}$. If $A=X$ then $X$ is called a {\it $\lambda$-separable space}.
\end{definition}

\begin{lemma}\label{2.4} Let $X$ be a Tychonoff $\lambda$-separable space for some family $\lambda$ of subsets $X$.  Then $C_{\lambda,\rho}(X)$ is submetrizable.
\end{lemma}

\begin{proof} Let $\lambda_1$ be a countable subfamily of
$\lambda$ such that $X=\overline{\bigcup\{B:B\in
\lambda_1\}}$. Let $X_1=\bigcup \{B:B\in \lambda_1\}$ and
$\lambda_2=\{A\cap X_1:A\in \lambda\}$. It is obvious that the mapping
$\pi_{X_1}:C_{\lambda,\rho}(X)\rightarrow
C_{\lambda_2,\rho}(X_1)$ where $\pi_{X_1}(f)=f\upharpoonright X_1$ for every $f\in C_{\lambda,\rho}(X)$  is a condensation because $X_1$
is dense in $X$. On the other hand, $\lambda_1\subseteq\lambda_2$ and,
hence, the identity mapping
$e:C_{\lambda_1,\rho}(X)\rightarrow C_{\lambda_2,\rho}(X_1)$ is
a condensation, too. The mapping $f=e\circ\pi_{X_1}$ is a condensation from $C_{\lambda,\rho}(X)$ onto the metrizable space $C_{\lambda_1,\rho}(X_1)$ ($C_{\lambda_1,\rho}(X_1)$ is metrizable by Proposition 4.9 in \cite{ANO}).
\end{proof}

\begin{corollary}\label{2.5} Let $X$ be a Tychonoff $\lambda$--separable space for some family $\lambda$ of subsets of $X$. Then for every function $h\in
C_{\lambda,\rho}(X)$ there is a function
$f:C_{\lambda,\rho}(X)\rightarrow \mathbb{R}$ such that $f(h)=0$
and $f(h_1)>0$ for any $h_1\in C_{\lambda,\rho}(X)$.
\end{corollary}

\begin{proof} Let $g:C_{\lambda,\rho}(X)\rightarrow Z$ be a condensation of the space $C_{\lambda,\rho}(X)$ onto a metric
space $Z$. Let $h\in C_{\lambda,\rho}(X)$ and $x=g(h)$. Then $f=f_1\circ g$ where the mapping $f_1:Z \rightarrow \mathbb{R}$ such that $f_1(y)=d(x,y)$ for every $y\in Z$.
\end{proof}

\begin{proposition}\label{2.6} Let $X$ be a Tychonoff space and let
$\lambda$ be a cover of $X$. If $F$ is a bounded subset of
$C_{\lambda,\rho}(X)$ then for every $\lambda$-separable
$Y\subseteq X$ there exists $F_1\subseteq C_{\lambda,\rho}(X)$ such
that $F\subseteq F_1$ and
$\pi_Y(F_1)=\overline{\pi_Y(F)}^{C_{\lambda_Y,\rho}(Y)}$ is a
metrizable compact space.
\end{proposition}

\begin{proof} Since $Y$ is $\lambda$-separable,
there exists a condensation $g:C_{\lambda_Y,\rho}(Y)\rightarrow Z$ where
$Z$ is a metrizable space (Lemma \ref{2.4}). Since $Z$ is a metrizable space, it is a $\mu$-space with a countable pseudocharacter.
By Lemma \ref{2.2}, $C_{\lambda_Y,\rho}(Y)$ is a
$\mu$-space. The set $\pi_Y(F)$ is bounded in
$C_{\lambda_Y,\rho}(X)$ and, hence, it is relatively compact.
Since the compact set $\overline{\pi_Y(F)}^{C_{\lambda_Y,\rho}(Y)}$
condenses into $Z$, then it is metrizable. For every point $h\in
\overline{\pi_Y(F)}^{C_{\lambda_Y,\rho}(Y)}$ the following condition holds:
$\pi^{-1}_Y(h)\cap C_{\lambda,\rho}(X)\neq \emptyset$.

Indeed, if this was not the case for some $h_1$, then
there would be a continuous function $\phi
:C_{\lambda,\rho}(Y)\rightarrow \mathbb{R}$ such that $\phi
(h_1)=0$ and $\phi(h)>0$ for $h\neq h_1$ (Corollary \ref{2.5}). Putting $\phi_1=\frac{1}{x}$ for every $x\in (0,+\infty)$,
$\phi_1:(0,+\infty)\rightarrow \mathbb{R}$ we would have a continuous
function $\phi_1\circ \phi\circ
\pi_Y:C_{\lambda,\rho}(X)\rightarrow \mathbb{R}$ which is unbounded
on $F$. Thus, we have $F_1=\bigcup \{\pi^{-1}_Y(h): h\in
\overline{\pi_Y(F)}^{C_{\lambda_Y,\rho}(Y)}\}$.
\end{proof}

\medskip

If $\lambda$ is a family of subsets of a topological
space $X$, then the family
of all countable unions of elements of $\lambda$ will be denoted by $\sigma\lambda$.

\medskip
 Let $X$ be a topological space, $\lambda\subseteq 2^X$. Let
$Q=\{B\subseteq X:\, \overline{A}\cap B$ is closed in
$\overline{A}$ for all $A\in\lambda\}$. Let $X_{\lambda}$ be a set
$X$ with the topology $\tau=\{X\setminus B:\, B\in Q\}$.  Further, we will call such
space a {\it $\lambda$-leader} of $X$.

Let $X$ be a Tychonoff space, $\lambda\subseteq 2^X$. Let
$X_{\tau\lambda}$ be a Tychonoff modification of the $\lambda$-leader
$X_{\lambda}$ of the space $X$.  Further, we will
call such space {\it $\lambda_f$-leader} of $X$ and denote by
$X_{\tau\lambda}$ (see more about $\lambda$- and $\lambda_f$-leaders of $X$ in \cite{ANO}).

\begin{theorem}\label{2.7} Let $X$ be a Tychonoff $AV_{\sigma \lambda}$-space for some cover $\lambda$ of $X$. Then $C_{\lambda,\rho}(X)$
is a $\mu$-space.
\end{theorem}

\begin{proof} Let $e:X_{\tau\lambda}\rightarrow X$ be a natural condensation of the $\lambda_f$-leader $X_{\tau\lambda}$ on
$X$. Denote by $Z=C_{e^{-1}(\lambda),\rho}(X_{\tau\lambda})$. Let
$F$ be a bounded subset of $C_{\lambda,\rho }(X)$,
then $F_0=\overline{F}^{Z}$ is compact because $Z$ is a complete uniform space (a $\mu$-space).

It suffices to show that $F_0\subset C_{\lambda,\rho}(X)$. Let $f\in F_0$. Let us prove that $f\in C_{\lambda,\rho}(X)$.

Assume the contrary, i.e. that $f$ is a discontinuous function from $X$ to
$\mathbb{R}$. Then there is a closed set $B\subset
\mathbb{R}$ such that $A=f^{-1}(B)$ is not closed in $X$. Since $X$ is an $AV_{\sigma \lambda}$-space, there is a point $x_0\in \overline{A}\setminus A$ for
which player TWO has a winning strategy in the game $AV_{\sigma \lambda}(X)$. Put $\epsilon
=\rho (f(x_0),B)$. Since $B$ is closed, $\epsilon
>0.$

The fact that $f\in F_0$ implies that there is a function $f_1\in F$
such that $|f(x_0)-f_1(x_0)|<\frac{1}{2}$. Since $\lambda$ is a cover of $X$, $\{x_0\}\in \lambda$ (Remark \ref{2.1}).

Let $V_1=\{x\in X:
|f(x_0)-f_1(x)|<\frac{1}{2}$. Then $V_1$ is a neighborhood of $x_0$.
Let $S_1\subseteq A$ and $T_1=\bigcup\{A^{i}_1\in \lambda :i\in
\mathbb{N}\}\in \sigma \lambda$ be a set chosen by player TWO
according to the winning strategy in the game $AV_{\sigma \lambda}(X)$. Let $P_1=A_1\in \lambda$.

Suppose that the functions $f_1,..., f_n$, the neighborhoods $V_1, ...,V_n$ of $x_0$, the subsets $S_1, ...,S_n$ of  $A$, the elements $T_1, ...,T_n$ of the family $\sigma \lambda$ such that $S_{j}\subseteq \overline{T_{j}}$ for every $j=1,...,n$ and
$T_{j}=\bigcup\{A^{i}_{j}\in \lambda: i\in \mathbb{N}\}\in
\sigma\lambda$ are constructed. Put $P_n=\bigcup\limits_{i=1}^n
\bigcup\limits_{j=1}^{n}A^{i}_{j}\cup\{x_0\}\in \lambda$ (see
Remark \ref{2.1}). Then there is a function $f_{n+1}$ such that
 $|f(x)-f_{n+1}(x)|<(\frac{1}{2})^{n+1}$ for every $x\in P_n$.
Let $V_{n+1}=\{x\in X:
|f(x_0)-f_{n+1}(x)|<(\frac{1}{2})^{n+1}\}$. Then $V_{n+1}$ is a neighborhood of
$x_0$. Let sets $S_{n+1}\subseteq A$ and
$T_{n+1}=\bigcup\{A^{i}_{n+1}\in \lambda : i\in \mathbb{N}\}\in
\sigma \lambda$ are chosen by player TWO according to the winning
strategy in the game $AV_{\sigma \lambda}(X)$.
Put $P_{n+1}=\bigcup\limits_{i=1}^{n+1}\bigcup\limits_{j=1}^{n+1}A^{i}_{j}\cup
\{x_0\}\in \lambda$.

Continuing this process, by induction we construct countable sets
$\{f_i: i\in \mathbb{N}\}$, $\{V_i:i\in \mathbb{N}\}$, $\{S_i: i\in
\mathbb{N}\}$ and $\{T_i: i\in \mathbb{N}\}$. Since $X$ is an $AV_{\sigma \lambda}$-space, there is a point
$x_{\omega }\in \overline{\bigcup\{S_i:i\in
\mathbb{N}\}}\cap(\bigcap\{V_i:i\in \mathbb{N}\})$.

Put $Y=\bigcup\{T_i:i\in
\mathbb{N}\}\cup\{x_0\}\cup\{x_{\omega}\}$. There is a function
$f_{\omega}\in C_{\lambda ,\rho}(X)$ such that $\pi_Y(f_{\omega
})\in \overline{\pi _Y(\{f_i:i\in
\mathbb{N}\})}^{C_{\lambda_Y,\rho}(Y)}$ (see Proposition \ref{2.6}).
Then two conditions hold:

(1) $f_{\omega }(x_{\omega })=f(x_0)$ because $x_{\omega }\in
\bigcap\{V_i:i\in \mathbb{N}\}$;

(2) $f_{\omega }(x)=f(x)$ for every $x\in \bigcup\{T_i:i\in
\mathbb{N}\}$.

  Let $T=(\bigcup\{T_i:i\in \mathbb{N}\})\setminus\{x\in X: |f(x_0)-f(x)|<\frac{3}{4 }\epsilon\}$.

By conditions (1) and (2), $|f_{\omega }(x_0)-f_{\omega }(x)|\ge
\frac{3}{4}\epsilon $ for every $x\in T$. If we get that
$x_{\omega }\in \overline{T}$, then we obtain a contradiction with a
continuity of $f_{\omega}$, and the theorem can be proven.

Indeed, let $O(x_{\omega })$ be an arbitrary neighborhood of the
point $x_{\omega }$. Then there is a point $x_1$ and $n\in
\mathbb{N}$ such that $x_1\in S_n\cap O(x_{\omega})$.

The set $H=T_n \cup \{x_1\}$ is $\lambda$-separable and, by Proposition \ref{2.6},
there is a function $f'\in C_{e^{-1}(\lambda),\rho}(X_{\tau\lambda})$
such that $\pi_H(f')=\pi_H(f)$. Since $f'$ is continuous,
there exists a point $x_2\in T_n$ such that
$|f'(x_1)-f'(x_2)|<\frac{1}{4}\epsilon $. But then

$|f(x_2)-f(x_0)|\ge |f(x_1)-f(x_0)|-|f'(x_1)-f'(x_2)|\ge \epsilon
-\frac{1}{4}\epsilon=\frac{3}{4}\epsilon$ , i.e. $x_2\in T\cap
O(x_{\omega })$.

\end{proof}

Theorem \ref{2.7} is a generalization of the Asanov-Velichko theorem in
\cite{AsVel}. It is interesting to note that in Theorem \ref{2.7} we can replace the assumption that $X$ is an $AV_{\lambda}$-space with the
assumption that $X$ is functionally generated by the family $\lambda$. This similarity forces us to find out the
relationship between classes of $AV_{\lambda}$-spaces and spaces that are functionally generated by the same family $\lambda$ of subsets.

\begin{proposition} There is a Tychonoff space $X$
functionally generated by a family $\lambda_c$ of all of its non-empty
countable subsets, which is not an $AV_{\lambda_c}$-space.
\end{proposition}

\begin{proof} Let $X=\prod\{R_{\alpha }:\alpha \in \Lambda\}$ where
$R_{\alpha }=\mathbb{R}$ for all $\alpha\in \Lambda$ where $|\Lambda|$ is an uncountable Ulam non-measurable cardinal \cite{Ul}. Then $X$ is functionally generated by a family $\lambda_c$ of all of its countable subsets (see \cite{Us}).
Let us prove that $X$ is not an $AV_{\lambda_c}$-space.

Let us introduce some notation: let $y=(y_{\alpha })\in
\prod\{R_{\alpha }: \alpha \in \Lambda\}$, $supp(y)=\{\alpha : y_{\alpha
}\neq 0\}$; a point $y$ for which $supp(y)=\emptyset$ will be denoted by ${\bf 0}$;
if $M$ is a subset of $X$ then $supp(M)=\bigcup\{supp(y): y\in
M\}$.

 Let $n\in \mathbb{N}$ and  $A=\{y\in X: |supp(y)|\le n$ and $(\forall \alpha :
y_{\alpha }\neq 0\Leftrightarrow y_{\alpha }=n)\}$.

Let us prove the following fact: there is no such point $x\in
\overline{A}\setminus A$ that player TWO has a winning
strategy in the game $AV_{\lambda _c}(X)$.

(1) Let us check that $\overline{A}\setminus A=\{{\bf 0}\}$. Let $x\in
X\setminus A$ and $x\neq {\bf 0}$, then three cases are possible:

\medskip

(a) $\exists \alpha^{0}\in supp(x)$ such that $x_{\alpha^0}$ is not an integer;

(b) $\exists \alpha^{0}_1,\alpha ^{0}_2\in supp(x)$ such that
$x_{\alpha^0_1}=i\neq j=x_{\alpha^0_2}$ where $i,j\in \mathbb{N}$;

(c) $\exists \alpha^{0}_1,\ldots ,\alpha ^{0}_{n+1}\in
supp(x)$ such that $x_{\alpha^0_1}=\ldots =x_{\alpha^0_{n+1}}=n.$

In all three cases, the point $x$ has
a neighborhood $V$ such that $V\cap A=\emptyset$.

\medskip

(a) $ V=\{y\in X: |x_{\alpha^0 }-y_{\alpha^0}|<\epsilon \}$ where
$\epsilon=\min \{|x_{\alpha^0}-n|: n\in \mathbb{N}\}$;

(b) $V=\{y\in X: |x_{\alpha^0_1}-y_{\alpha^0_1}|<\frac{1}{2}$
and $|x_{\alpha^0_2}-y_{\alpha^0_2}|<\frac{1}{2}\}$;

(c) $V=\{y\in X: |y_{\alpha^0
_i}-x_{\alpha^0_i}|<\frac{1}{2}$ for every $i=1,\ldots,n+1\}$.

  Thus, we have proved
that $\overline{A}\setminus A=\{{\bf 0}\}$.

\medskip

(2) Prove that player TWO does not have a winning strategy in the game $AV_{\lambda _c}(X)$ on the set $A$ and at the point ${\bf 0}$. Moreover, we
 prove that player ONE does have a winning strategy.

Let us check that the choice of
player TWO in every round of the game can only be a countable set
$S_i$. Note that $S_i$ contains at most
a countable number of points having the same $supp$. On the other hand, if $y\in \overline{T_i}$, then there is $y'\in
\overline{T_i}$ such that $supp(y)=supp(y')$. Since $T_i\in
\lambda_c$ and, hence, it is countable, $S_i$ is countable, too.

Now let the set $S_1$ be chosen by player TWO  at the first step of the
game. Let $H_1=supp(S_1)$. We renumber the set
$H_1=\{\alpha^{1}_i: i\in \mathbb{N}\}$. Take a
neighborhood of the point ${\bf 0}$ chosen by the player ONE, the set
$V_1=\{x\in X: |x_{\alpha^{1}_1}|<\frac{1}{2}\}$. Let the game already
pass $m-1$ steps and the set $S_m\subseteq A$ is chosen by the player TWO
 at the $m$-th step of the Asanov-Velichko game, let
$H_m=supp(S_m)=\{\alpha ^{m}_i:i\in \mathbb{N}\}$.

Let $V_m=\{x\in X:\forall i,\forall j:i\le j\le m,
|x_{\alpha^{i}_{j}}|<\frac{1}{2}\}$. And so on. At the end of the game
we get:  for every $x\in \bigcap \{V_m: m\in \mathbb{N}\}$  $|x_{\alpha
}|<\frac{1}{2}$ for any $\alpha \in H=\bigcup \{H_m:m\in \mathbb{N}\}$. On the other hand, for every $x\in
\overline{\bigcup \{S_m: m\in \mathbb{N}\}}$ there is $\alpha \in
H$ such that $x_{\alpha }\ge 1$. So $\overline{\bigcup
\{S_m: m\in \mathbb{N}\}}\cap (\bigcap \{V_m: m\in
\mathbb{N}\})=\emptyset$.
\end{proof}

\begin{proposition} There is a Tychonoff $AV_{\lambda_c}$-space $X$ where $\lambda_c$ is a family of all of its
non-empty countable subsets, but $X$ is not functionally generated by $\lambda_c$.
\end{proposition}

\begin{proof} Let $X=T(\omega_1+1)$ be the set of all
ordinals less than $\omega_1+1$ in the order topology (see
Example 3.1.27 in \cite{Eng}). The space $X$ is compact and, hence, is an $AV_{\lambda_c}$-space.

But $X$ is not functionally generated by the family $\lambda_c$, because the function

$$ f(\alpha)= \left\{
\begin{array}{lcr}
0, \ \ \ \ if \, \alpha <\omega_1, \\
1, \, \, \, \, if \, \, \, \, \alpha=\omega_1\\
\end{array}
\right.
$$

is discontinuous, but $f\upharpoonright S$ is continuous for any countable set $S$.
\end{proof}

To conclude this section, we present a theorem that reduces
the problem of compactness of subsets in $C_{\lambda,\mu}(X,Y)$ to the
case in $C_p(X,Y)$.

\medskip

A continuously invariant property ${\cal P}$ is of
{\it boundedness type} if and only if it implies boundedness, that is, if every
subspace $Y$ of $X$ with the property ${\cal P}$  is bounded in $X$ (Proposition 2.15 in \cite{Arch2}).

\begin{theorem}\label{2.11} Let $X$ be a Hausdorff topological
space, let $\lambda$ be a cover of $X$, and let $Y$ be a complete
Hausdorff uniform space with uniformity $\mu$.
Then, if ${\cal P}$ is a boundedness type property such that
any subset of $C_p(X,Y)$ that has this property
is relatively compact in $C_p(X,Y)$, then every subset of
$C_{\lambda ,\mu}(X,Y)$ with the property ${\cal P}$, is
relatively compact in $C_{\lambda,\mu}(X,Y)$.
\end{theorem}

\begin{proof} Let $F\subseteq C_{\lambda,\mu}(X,Y)$
have the property ${\cal P}$ and let $e:X_{\tau\lambda}\rightarrow X$
be a natural condensation of the $\lambda_f$-leader $X_{\tau\lambda}$
of $X$ onto  $X$. By Corollary 4.3 in \cite{ANO}, the map
$e^{\#}:C_{\lambda,\mu}(X,Y)\rightarrow
C_{e^{-1}(\lambda),\mu}(X_{\tau\lambda},Y)$ is an embedding, and the set
$e^{\#}(F)$ has the property ${\cal P}$ (moreover, $e^{\#}(F)$ is bounded).
Since $C_{e^{-1}(\lambda),\mu}(X_{\tau\lambda},Y)$ is a complete
uniform space (see Proposition 4.7 in \cite{ANO}) (hence, a $\mu$-space),  the closure of the set $e^{\#}(F)$ in the
space $C_{e^{-1}(\lambda),\mu}(X_{\tau\lambda},Y)$ is
compact.

If we now prove that $e^{\#}(\overline{F}^{C_{\lambda,\mu}(X,Y)})=
\overline{e^{\#}(F)}^{C_{e^{-1}(\lambda),\mu}(X_{\tau\lambda},Y)}$
then, due to the fact that the mapping $e^{\#}$ is an embedding, we get
that $\overline{F}^{C_{\lambda,\mu}(X,Y)}$ is compact.

Note that the set $F$ has the property ${\cal P}$ in
space $C_p(X,Y)$. Since the identity mapping
$id_{C(X,Y)}:C_{\lambda,\mu}(X,Y)\rightarrow C_p(X,Y)$ is continuous, the space $C_p(X,Y)$ has a weaker topology than the topology of
$C_{\lambda,\mu}(X,Y)$ and, by the hypotheses of the theorem, $F$ is relatively compact in $C_p(X,Y)$. Since $\overline{F}^{C_p(X,Y)}$ is compact, the set
$e^{\#}(\overline{F}^{C_p(X,Y)})$ is compact in $C_p(X_{\tau\lambda},Y)$. On the other hand,  since the
mapping $e^{\#}:C_p(X,Y)\rightarrow C_p(X_{\tau\lambda},Y)$
is continuous (see \cite{Arch3}, Corollary 2.8), $e^{\#}(\overline{F}^{C_p(X,Y)})=
\overline{e^{\#}(F)}^{C_p(X_{\tau\lambda},Y)}$.

Since the identity mapping
$id_{C(X_{\tau\lambda},Y)}:C_{e^{-1}(\lambda),\mu}(X_{\tau\lambda
},Y)\rightarrow C_p(X_{\tau\lambda},Y)$ is continuous (see \cite{Arch3}),
$\overline{e^{\#}(F)}^{C_{e^{-1}(\lambda),\mu}(X_{\tau\lambda},Y)}\subseteq
\overline{e^{\#}(F)}^{C_p(X_{\tau\lambda},Y)}$, i.e.
$\overline{e^{\#}(F)}^{C_{e^{-1}(\lambda),\mu}(X_{\tau\lambda},Y)}\subseteq
e^{\#}(C_{\lambda,\mu}(X,Y))$.
\end{proof}

Theorem \ref{2.11} implies the following corollaries.

\begin{corollary} Let $X$ be a Tychonoff space and let $\lambda$ be a cover of $X$. If $X$ is an $AV_{\gamma_c}$-space or functionally generated by a family $\gamma_c$ of all of its subspaces that have a dense
$\sigma$-countably pracompact subset, then each pseudocompact subset in $C_{\lambda ,\rho}(X)$ is relatively compact.
\end{corollary}

\begin{proof} It suffices to refer to Theorems \ref{2.11} and \ref{1.10}
of this paper and Theorem 8.1 in \cite{Arch2}.
\end{proof}

\begin{corollary}\label{2.13} Let $X$ be a Tychonoff space and let $\lambda$ be a cover of $X$. If $X$ is an $AV_{\gamma_p}$-space or functionally generated by a family $\gamma_p$ of all of its subspaces that have a dense
$\sigma$-pseudocompact subset, then each countably pracompact subset in $C_{\lambda ,\rho}(X)$ is relatively compact.
\end{corollary}

\begin{proof} It suffices to refer to Theorems \ref{2.11} and \ref{1.11}
of this paper and Theorem 8.3 in \cite{Arch2}.
\end{proof}

\begin{corollary}\label{2.14} Let $X$ be a Tychonoff space and let $\lambda$ be a cover of $X$. If $X$ is an $AV_{\gamma}$-space or functionally generated by a family $\gamma$ of all of its countable subsets, then the space $C_{\lambda,\rho}(X)$ is a
$\mu$-space.
\end{corollary}

\begin{proof} It suffices to refer to Theorem \ref{2.11} of this paper,
Theorem 1 in \cite{AsVel} and Theorem 2.16 in \cite{Arch2}.
\end{proof}

To indicate the span of the classes of spaces that cover these corollaries, we indicate that Corollary \ref{2.14} applies to

\medskip

(1) $q$-spaces in the sense E. Michael \cite{Mic}, in particular, countably
compact spaces;

\medskip

(2) quasi-$k$--spaces in the sense of Ju. Nagata \cite{Nag}, in particular,
$k$-spaces;

\medskip
(3) spaces of countable tightness, in particular, sequential and
Frechet--Urysohn spaces;

\medskip

(4) locally separable spaces;

\medskip

(5) spaces of countable functional tightness (see \cite{arch} and \cite{AsVel}).

\medskip

Corollary \ref{2.13} also applies to $\sigma$-pseudocompact
spaces.

\section{Grothendieck's theorem for $C_\lambda(X)$}

The purpose of this section is to prove a theorem
analogous to Theorem \ref{2.11}. For this we need two lemmas.

\begin{lemma} Let $X=\lim\limits_\leftarrow \{X_{\alpha
},\pi^{\alpha1}_{\alpha 2},A\}$ be the inverse limit of the system $\{X_{\alpha
},\pi^{\alpha1}_{\alpha 2},A\}$ of Tychonoff spaces $X_{\alpha }$, and for every $\alpha$ let any subset of  $X_{\alpha }$ with the
boundedness property of type ${\cal P}$ be relatively compact. Then any
$B\subseteq X$ with the property ${\cal P}$ is relatively compact.
\end{lemma}

\begin{proof} Let $B\subseteq X$ have the property
${\cal P}$, then for every $\alpha \in A$, $\pi_{\alpha }(B)$
has the property ${\cal P}$ and, therefore,
$\overline{\pi_{\alpha}(B)}^{X_{\alpha}}$ is compact. Obviously $\overline{B}^{X}=\lim \limits_{\leftarrow}\{\overline{\pi_{\alpha}(B)}^{X_{\alpha}},\pi^{\alpha 1}_{\alpha 2},A\}$ is compact because it is a closed subset of a compact set.
\end{proof}

The following result was proved by D. Preiss and P. Simon.

\begin{proposition}(Theorem 5 in  \cite{PS})\label{3.2}
Let $X$ be an Eberlein compactum, and let $x$ be a non-isolated point of $X$. Then there exists a sequence $\{U_n: n\in \mathbb{N}\}$ of open sets in $X$ which converges to $x$.
\end{proposition}

\begin{lemma}\label{3.3} Let $X$ be a regular space. Let $f:X\rightarrow Y$
be a condensation and let the closure of any pseudocompact $A\subseteq Y$ be an
Eberlein compactum, then the closure of any pseudocompact
$B\subseteq X$ is an Eberlein compactum.
\end{lemma}

\begin{proof} Let $B$ be a pseudocompact subset of $X$. Since
$f(B)$ is pseudocompact,  $A=\overline{f(B)}$ is an Eberlein compactum. Then $f(B)$ is itself an Eberlein compactum, and is, moreover, closed in $Y$ (see Theorem IV.5.4 in \cite{arch}).

Let us now prove that $f|_{B}$ is a closed mapping (and,
hence, a homeomorphism). Let $C$ be an arbitrary closed
subset of $B$. Since $X$ is regular, there is a family
$\gamma$ of open sets in $B$ such that
$C=\bigcap\{\overline{U}^{B}: U\in \gamma\}$. Since $\overline{U}^B$ is pseudocompact,
the set $f(\overline{U}^B)$ is pseudocompact \cite{Col}, and, hence, $f(\overline{U}^B)$ is
closed in $Y$. Therefore $f(C)=f(\bigcap\{\overline{U}^B:U\in
\gamma\})=\bigcap\{f(\overline{U}^B):U\in \gamma\}$
is a closed set in $Y$. Thus, we have proved that
$f(B)$ is an Eberlein compactum and $f|_{B}$ is a homeomorphism.
Hence, $B$ is closed and $B$ is an Eberlein compactum.
\end{proof}

\begin{corollary} If a Tychonoff space $X$ contains a
dense $\sigma$-countably pracompact subspace, and $\lambda$ is
a family of subsets of $X$ containing all finite subsets of
$X$, then every pseudocompact subset of $C_{\lambda}(X)$ is an Eberlein compactum.
\end{corollary}

\begin{theorem}\label{3.5} Let $X$ be a Tychonoff space. Let $\lambda$
be a family containing all non-empty finite subsets of
$X$, such that it is closed under finite intersections and unions.
Let ${\cal P}$ be a boundedness type property implying
pseudocompactness such that for each $A\in \lambda$ any
subset of the space $C_p(\overline{A})$ with
the property ${\cal P}$, is an Eberlein compactum.
Then any subset of $C_{\lambda}(X)$ with the property ${\cal P}$ is relatively compact provided that any subset of $C_p(X)$ with the property ${\cal P}$ is
relatively compact.
\end{theorem}

\begin{proof} Let $e:X_{\tau\lambda}\rightarrow X$
be a natural condensation of a $\lambda_f$-leader $X_{\tau\lambda}$ of
$X$ onto $X$. By Theorem 5.1 in \cite{ANO}, the map
$e^{\#}:C_{\lambda}(X)\rightarrow
C_{e^{-1}(\lambda)}(X_{\tau\lambda})$ is an embedding.

Let $F\subseteq C_{\lambda}(X)$ have the property ${\cal P}$.
Denote by $F_1$ the closure of the set $e^{\#}(F)$ in the
space $C_{e^{-1}(\lambda)}(X_{\tau\lambda})$. Let us prove that
$F_1$ is compact.

By Theorem 5.2 in \cite{ANO}, $C_{e^{-1}(\lambda)}(X_{\tau\lambda})$ is homeomorphic to the inverse limit of the system $S_{*}(X,\lambda, \mathbb{R})$, and  the set $\pi_{\overline{A}}(F_1)$ has the property ${\cal P}$ in
$C_{\lambda_{\overline{A}}}(\overline{A})$ for every $A\in \lambda$. On the other hand,
$C_{\lambda_{\overline{A}}}(\overline{A})$ condenses onto $C_p(\overline{A})$. Therefore, by Lemma \ref{3.3}, the closure of the set $\pi_{\overline{A}}(F_1)$ into
$C_{\lambda_{\overline{A}}}(\overline{A})$ is compact.
Hence, $F_1$ is compact as a closed subset of the inverse limit of the system consisting of compact sets.

Note that the set $F$ has the property ${\cal P}$ in
space $C_p(X)$ (which has a weaker topology than the topology in
$C_{\lambda}(X)$), and, by assumption, is relatively compact.
Therefore, $e^{\#}(\overline{F}^{C_p(X)})$ is compact and
closed in $C_p(X_{\tau\lambda})$. On the other hand,
$e^{\#}(\overline{F}^{C_p(X)})\subseteq
\overline{e^{\#}(F)}^{C_p(X_{\tau\lambda})}$ due to continuity
$e^{\#}$ (\cite{Arch3}, Corollary 2.8.). But then
$e^{\#}(\overline{F}^{C_p(X)})=\overline{e^{\#}(F)}^{C_p(X_{\tau\lambda
})}$.

Since $F_1\subseteq \overline{e^{\#}(F)}^{C_p(X_{\tau\lambda
})}=e^{\#}(\overline{F}^{C_p(X)})$,  $F_1\subseteq
e^{\#}(C_{\lambda }(X))$ and, hence, $F_1$ is homeomorphic to
$\overline{F}^{C_{\lambda}(X)}$ which implies that $F$ is
relatively compact in $C_{\lambda}(X)$.
\end{proof}

\begin{corollary} Let $X$ be a Tychonoff space and let $\lambda$ be a cover of $X$ containing all finite subsets of $X$. If $X$ is an $AV_{\gamma}$-space or is functionally generated by a family $\gamma$ of all of its subspaces that have a dense
$\sigma$-countably pracompact subset  and $\lambda\subseteq\gamma$, then each pseudocompact subset in $C_{\lambda}(X)$ is relatively compact.
\end{corollary}

\begin{proof} It suffices to refer to Theorems \ref{3.5} and \ref{1.10}
of this paper and Theorem 1.1 in \cite{Arch2}.
\end{proof}

\begin{corollary}\label{2.13} Let $X$ be a Tychonoff space and let $\lambda$ be a cover of $X$ containing all finite subsets of $X$. If $X$ is an $AV_{\gamma}$-space or functionally generated by a family $\gamma$ of all of its subspaces that have a dense
$\sigma$-pseudocompact subset, then each countably pracompact subset in $C_{\lambda}(X)$ is relatively compact.
\end{corollary}

\begin{proof} It suffices to refer to Theorems \ref{3.5} and \ref{1.11}
of this paper and Theorem 8.3 in \cite{Arch2}.
\end{proof}


\bibliographystyle{model1a-num-names}
\bibliography{<your-bib-database>}

\end{document}